\documentclass{amsart}
\usepackage{amsmath}
\usepackage{graphicx}
\usepackage{amsfonts}
\usepackage{amssymb}

\newtheorem{theorem}{Theorem}
\theoremstyle{plain}

\newtheorem{proposition}{Proposition}

\numberwithin{equation}{section}

\begin{document}
\title{On derived categories and derived functors}
\author{Samson Saneblidze}
\address{A. Razmadze Mathematical Institute\\
Department of Geometry and Topology\\
M. Aleksidze st., 1\\
0193 Tbilisi, Georgia} \email{sane@rmi.acnet.ge}

\thanks{}



\date{}

\begin{abstract}
For an abelian category,  a category equivalent to its derived category is constructed
by means of specific projective (injective) multicomplexes, the so-called  homological
resolutions.

\end{abstract}

\maketitle

\section{Introduction}

The derived category $D(\mathcal{A})$ of an abelian  category $\mathcal{A}$ was
introduced by Verdier in 1963, see \cite{Verdier} and  \cite{Verdier2}.  It was defined
as the localization of the category of unbounded chain complexes with respect to
quasi-isomorphisms. The existence of $D(\mathcal{A})$ creates set-theoretical problems.
Verdier proved the existence of $D(\mathcal{A})$  only in the case when $\mathcal{A}$
has finite global dimension. Later existence of $D(\mathcal{A})$  was established by
Spaltenstein \cite{Spaltenstein} in the case when $\mathcal{A}$ is the category of
modules over a ring, or more generally category of modules over a sheaf of rings. The
first case was also considered by Hovey in \cite{Hovey}. Recently the existence of
$D(\mathcal{A})$  was proved in the case when $\mathcal{A}$  is a Grothendieck
category, see for example \cite{Alonso}.

In the present paper we prove the existence of $D(\mathcal{A})$ in the case when
$\mathcal{A}$  has enough projectives and countable coproducts. By duality the same is
true provided $\mathcal{A}$ has enough injectives and countable products.

The essential part of the paper was in fact written about 15 years ago when the author
was visiting the Heidelberg University. As it is partially reviewed above in the
meantime there appeared various kinds of descriptions of derived categories, however,
the decision to write this paper is motivated by reasons  mentioned above and continued
below: Our approach uses a theory of \emph{multicomplexes} and  emphasis a r\^{o}le of
the homology of differential graded objects (bounded or unbounded) on the additive
level; nowadays multicomplexes are considered to be endowed with multiplicative or
higher order operations that measure certain standard relations up to homotopy (see,
for example, \cite{hal-sta}, \cite{hueb}, \cite{sane}, \cite{sane2}). It should be
noted that such enriched homological multicomplexes are candidates to be (co)fibrant
objects in the appropriated closed model category, since the analogs of Proposition
\ref{Whitehead} below (compare Proposition 3  in \cite{sane2}). So that it is expected
to use them for homotopy classification problems behind the rational homotopy theory
too.

I  thank T. Pirashvili for useful discussions. I also thank  J. Huebschmann for
comments.

\vspace{0.2in}
\section{The main result}

A chain map between unbounded chain complexes which induces an isomorphism in homology
(i.e. a quasi-isomorphism) is not a homotopy equivalence even   each complex consists
of projective objects of an abelian category $\mathcal{A},$ and also an additive
functor, such as $Hom$ and $\otimes$ one,  does  not preserve quasi-isomorphisms
\cite{Dold}. There are various kinds of restrictions on chain complexes that guarantees
quasi-isomorphisms to be homotopy equivalences (see \cite{Dold}, \cite{Spaltenstein},
\cite{Hovey}, \cite{Gelfand-Manin}).

Here we  consider the other kind of restriction  by introducing special projective
(injective) chain complexes, \emph{homological multicomplexes}. In particular, by means
of these complexes we can describe the derived category $D(\mathcal{A})$ and to
construct the derived functor for an additive functor mentioned in the introduction.

Usually the above restrictions are evoked to start inductively from the first
non-trivial component of a (bounded) chain complex. In our case, the induction relies
on  a specific filtration of the total complex of a multicomplex involving all  (total)
degrees simultaneously (compare \cite{Dold}, \cite{Hovey}). On the other hand, given a
chain complex $A$ over $\mathcal{A},$ one considers in the theory of derived category
projective (injective) replacements $C\rightarrow A$($A\rightarrow C$)  of $A$, i.e.
quasi-isomorphisms with $C$ consisting  of projective (injective) components from
$\mathcal{A}.$ We show that for each chain complex $A$ there is a multicomplex $C$ such
that its total complex is quasi-isomorphic to $A$ and $C^{\ast,j}$ is a projective
resolution of the cohomology $H^j(A)$ for each $j\in \mathbb{Z}.$ So that among
projective replacements of $A$ mentioned above, \emph{the homological resolution} $C$
could be chosen small as possible.

In order to state our main theorem below we choose the language of \emph{projective}
objects (the case of injective objects is entirely dual).

Given an abelian category $\mathcal{A}$ with countable coproducts,  a
\emph{multicomplex} over $\mathcal{A}$ is a bigraded object $C^{\ast,\ast}=\{C^{i,j}\}
_{i,j\in \mathbb{Z}}$ together with morphisms $ d^r: C^{i,j}\rightarrow
C^{i+r,j-r+1},\,r\geq 0,$ such that $\sum_{p+q=n}d^pd^q=0$ for each $n\geq 0.$ The
\emph{total} complex of $C^{\ast,\ast}$ is the chain complex $(Tot(C),d^{\ast})$ with
$$Tot(C)^n=\bigoplus_{i+j=n}C^{i,j}\ \ \text{and}\ \
d^{\ast}=d^0+d^1+\dotsb +d^r+\dotsb .$$ In particular,  $d^1d^1=0$ when  $d^0=0.$  A
multicomplex $(C^{\ast,\ast},d^{\ast})$ is called \emph{homological} if $d^0=0,$
$C^{i,\ast}=0$ for $ i>0$ and $H^{i}(C^{i,*},d^1)=0$ for  $ i<0 .$ A multicomplex
$(C^{\ast,\ast},d^{\ast})$ is called \emph{projective} if each $C^{i,j}$ is a
projective object of $\mathcal{A}.$ A \emph{column (resolution) filtration} of a
multicomplex $(C^{\ast,\ast},d^{\ast})$ is a sequence $\{C_{(k)}\}_{k\leq 0}$ with
$C_{(k)}=\bigoplus_{k\leq i\leq 0} C^{i,\ast}.$

A {\em multicomplex map} $f:A\rightarrow B$ between two multicomplexes $A$ and $B$ is a
chain map of total degree zero that preserves the column (resolution) filtration, i.e.
$fTot(A)^n\subset Tot(B)^n$ and $fA_{(k)}\subset B_{(k)};$ so that $f$ has the
components $f=f^0+\dotsb +f^i+\dotsb $ with  $ f^i:A^{s,t}\rightarrow B^{s+i,t-i}.$ A
{\em homotopy} between two maps $f,g:A\rightarrow B $ of multicomplexes is a chain
homotopy $s:A\rightarrow B$ of total degree $-1$ that lowers  the column filtration by
1, i.e.  $sTot(A)^n\subset Tot(B)^{n-1}$ and $sA_{(k)}\subset B_{(k-1)}.$

Note that, unlike standard bicomplexes, in a homological multicomplex we have no
vertical differentials; this fact together with the acyclicity with respect to the
horizontal differential $d^1$ guarantees the spectral sequence arising from the column
filtration to be collapsed; in particular, the other components $d^r,r\geq 2,$ have no
action to change the cohomology non-isomorphically; in other words, when $d^r$ varies
 in $Hom(C^{i,j}, C^{i+r,j-r+1})$ for $r\geq 2$ one
obtains multicomplexes with isomorphic
 cohomologies (see Fig. 1).

Let $K(\mathcal{A})$ be the category whose objects are chain complexes over
$\mathcal{A}$ and morphisms are homotopy classes of maps denoted by $[-,-],$
 $K_{\mathcal{M}}(\mathcal{A})$ be
the category whose objects are  multicomplexes over $\mathcal{A}$ and morphisms are
homotopy classes of multicomplex maps denoted by $[-,-]_{\mathcal{M}},$ while
$K_{\mathcal{P}}(\mathcal{A})$ be the (sub)category whose objects are homological
projective multicomplexes over $\mathcal{A}$ and morphisms are homotopy classes of maps
denoted by $[-,-]_{\mathcal{P}}$.

Recall that the \emph{derived category} $D(\mathcal{A})$ of $\mathcal{A}$ is defined as
 the category obtained from $K(\mathcal{A})$ by inverting the class of
quasi-isomorphisms \cite{Gelfand-Manin}, \cite{Iversen}, \cite{Verdier}, and let
$Q:K(\mathcal{A})\rightarrow D(\mathcal{A})$ be the localization functor. Let
$\kappa:K_{\mathcal{P}}(\mathcal{A})\rightarrow K(\mathcal{A})$ be the functor defined
by the following obvious proposition.

\begin{proposition}\label{inclusion} Given a multicomlex $C$ and a morphism  $\mathfrak{f} :C\rightarrow C^{\prime}$
in $K_{\mathcal{P}}(\mathcal{A}),$
  the assignments $C\rightarrow
Tot(C)$ and $\mathfrak{f} \rightarrow [Tot(f)]$   define a functor
$\kappa:K_{\mathcal{P}}(\mathcal{A})\rightarrow K(\mathcal{A})$ for a representative
$f$ of $\mathfrak{f}.$

\end{proposition}
\begin{proof}
 First remark that the
assignments  $C\rightarrow Tot(C)$ and $f \rightarrow Tot(f)$ define a functor from the
category of multicomplexes and multicomplex maps to the category of chain complexes and
chain maps over $\mathcal{A}.$ Now if $f,g:C\rightarrow C'$ are two chain homotopic
maps of multicomplexes  $f\underset{s}\simeq g,$ then clearly $s$ induces a map
$Tot(s):Tot(C)\rightarrow Tot(C')$ such that $Tot(f)\underset{Tot(s)}\simeq Tot(g).$
\end{proof}

  Consider the functor
\[\iota: K_{\mathcal{P}}(\mathcal{A}) \rightarrow D(\mathcal{A})\]
  obtained as
the composition $ K_{\mathcal{P}}(\mathcal{A})\overset{\kappa}\rightarrow
K(\mathcal{A})\overset{Q}\rightarrow D(\mathcal{A}). $

The main statement here is the following

\begin{theorem}\label{main} If an abelian category $\mathcal{A}$ has enough
projectives and countable coproducts, then the functor $\iota:
K_{\mathcal{P}}(\mathcal{A})\rightarrow D(\mathcal{A})$ is an equivalence of
categories.
\end{theorem}

This theorem relies on the following 'Whitehead (or Adams-Hilton) type' proposition
that has an independent interest. Given a chain complex $A,$ we consider it as bigraded
via $A^{0,\ast}=A^{\ast}$ and $A^{i,\ast}=0$ for $i\neq 0,$ and then regard
$K(\mathcal{A})$ as the subcategory of $K_{\mathcal{M}}(\mathcal{A}).$

\begin{proposition}\label{Whitehead}
Let     $f:A\rightarrow B$  be a quasi-isomorphism in $K_{\mathcal{M}}(\mathcal{A})$
where $A $ or $ B$ is a chain complex or a homological multicomlex  over $\mathcal{A}.$
If $C$ is a homological projective multicomplex then  the induced map
$f_{_{\#}}:[C,A]_{\mathcal{M}}\rightarrow [C,B]_{\mathcal{M}}$ is a bijection.
\end{proposition}

\vspace{0.2in}

\section{Proof of Theorem 1}

Given a chain complex $(A,d)$ from $K(\mathcal{A}),$ its \emph{homological resolution}
is a homological projective multicomlex $(C^{\ast,\ast},d^{\ast})$ with a multicomlex
map $\phi:C\rightarrow A$ inducing
 a quasi-isomorphism
$Tot(\phi):(Tot(C),d^{\ast})\rightarrow  (Tot(A),d) =(A,d).$ In particular,
$(C^{\ast,j},d^1)$ forms a projective resolution of the object $H^j(A)$ for all $j\in
\mathbb{Z}$ (see Fig. 1). \vspace{0.2in}

\unitlength=1.00mm \special{em:linewidth 0.4pt} \linethickness{0.4pt}
\begin{picture}(87.33,38.00)
\put(30.00,5.00){\vector(1,0){10.00}} \put(40.00,5.00){\vector(1,0){10.00}}
\put(50.00,5.00){\vector(1,0){10.00}} \put(60.00,5.00){\vector(1,0){10.33}}
\put(70.33,5.00){\vector(1,0){9.67}} \put(30.00,15.33){\vector(1,0){10.00}}
\put(40.00,15.33){\vector(1,0){10.00}} \put(50.00,15.33){\vector(1,0){10.00}}

\put(60.00,15.33){\vector(1,0){10.00}}

\put(70.33,15.33){\vector(1,0){9.55}} \put(30.00,25.00){\vector(1,0){10.00}}
\put(40.00,25.00){\vector(1,0){10.00}}

\put(50.00,25.00){\vector(1,0){10.00}} \put(60.00,25.00){\vector(1,0){10.33}}

\put(70.33,25.00){\vector(1,0){9.67}}

 \put(30.00,35.33){\vector(1,0){10.00}}
\put(40.00,35.33){\vector(1,0){10.00}} \put(50.00,35.33){\vector(1,0){10.00}}
\put(60.00,35.33){\vector(1,0){10.33}} \put(70.33,35.33){\vector(1,0){9.67}}

\put(40.00,35.53){\vector(3,-2){30.23}} \put(40.00,35.33){\vector(2,-1){20.20}}
\put(50.00,35.33){\vector(2,-1){20.37}} \put(50.00,15.33){\vector(2,-1){20.37}}

\put(75.00,7.00){\makebox(0,0)[cc]{$\rho$}}
\put(75.00,17.00){\makebox(0,0)[cc]{$\rho$}}
\put(75.33,27.00){\makebox(0,0)[cc]{$\rho$}}
\put(75.33,37.33){\makebox(0,0)[cc]{$\rho$}}

\put(64.00,30.70){\makebox(0,0)[cc]{$_{d^2}$}}
\put(53.80,30.70){\makebox(0,0)[cc]{$_{d^2}$}}
\put(66.00,20.50){\makebox(0,0)[cc]{$_{d^3}$}}
\put(66.00,9.75){\makebox(0,0)[cc]{$_{d^2}$}}

\put(65.00,37.00){\makebox(0,0)[cc]{$_{d^1}$}}
 \put(55.00,37.00){\makebox(0,0)[cc]{$_{d^1}$}}
\put(45.00,37.00){\makebox(0,0)[cc]{$_{d^1}$}}
\put(64.60,26.20){\makebox(0,0)[cc]{$_{d^1}$}}
\put(64.00,16.67){\makebox(0,0)[cc]{$_{d^1}$}}
\put(64.33,6.33){\makebox(0,0)[cc]{$_{d^1}$}}

\put(55.00,16.67){\makebox(0,0)[cc]{$_{d^1}$}}
\put(44.67,16.67){\makebox(0,0)[cc]{$_{d^1}$}}
\put(54.67,6.33){\makebox(0,0)[cc]{$_{d^1}$}}
\put(44.67,6.33){\makebox(0,0)[cc]{$_{d^1}$}}

\put(80.33,35.33){\circle*{0.67}} \put(80.00,25.00){\circle*{0.67}}
\put(80.00,15.33){\circle*{0.67}} \put(80.00,5.00){\circle*{0.67}}
\put(70.33,5.00){\circle*{1.33}} \put(70.33,15.33){\circle*{1.33}}
\put(70.33,25.00){\circle*{1.33}} \put(70.33,35.33){\circle*{1.33}}
\put(60.00,35.33){\circle*{1.33}} \put(50.00,35.33){\circle*{1.33}}
\put(39.67,35.33){\circle*{1.33}} \put(60.00,25.00){\circle*{1.33}}
\put(50.00,25.00){\circle*{1.33}} \put(40.00,25.00){\circle*{1.33}}
\put(40.00,15.33){\circle*{1.33}} \put(50.00,15.33){\circle*{1.33}}
\put(60.00,15.33){\circle*{1.33}} \put(60.00,5.00){\circle*{1.33}}
\put(50.00,5.00){\circle*{1.33}} \put(40.00,5.00){\circle*{1.33}}

\put(88.83,5.00){\makebox(0,0)[cc]{$H^{j-1}(A)$}}
\put(87.33,15.33){\makebox(0,0)[cc]{$H^j(A)$}}
\put(88.83,25.00){\makebox(0,0)[cc]{$H^{j+1}(A)$}}
\put(88.83,35.33){\makebox(0,0)[cc]{$H^{j+2}(A)$}}
\end{picture}
\vspace{-0.2in}
$$
\hspace{-0.1in} _{\cdots\rightarrow C^{-3,\ast}\overset{d^1}\rightarrow C^{-2,\ast}
\overset{d^1}\rightarrow  C^{-1,\ast}\overset{d^1}\rightarrow\, C^{0,\ast}
\overset{\rho} \rightarrow \,\,H^{\ast}(A)}$$

\vspace{0.1in} \begin{center}{Figure 1. A fragment of a homological resolution.}
\end{center}
\vspace{0.2in}
\begin{proposition}\label{resolution}
If an abelian category $\mathcal{A}$ has enough projectives and countable coproducts,
then any chain complex $A$ of $K(\mathcal{A})$ has a homological resolution
$\phi:C\rightarrow A.$

\end{proposition}
\begin{proof}
First choose a projective resolution $\rho:(C^{\ast,j},d^1)\rightarrow H^j(A)$
  of  $H^j(A)$ for each $j\in \mathbb{Z}$ so that
$d^1:C^{i-1,\ast}\rightarrow
  C^{i,\ast}$ for $i\leq 0.$
Consider the epimorphism
\[\rho^0=\rho|_{C^{0,\ast}}:C^{0,\ast}\rightarrow H^{\ast}(A).
\]
Factor it through cocycles $C^{0,\ast}\overset{\phi'}\rightarrow
ZA^{\ast}\overset{\nu}\rightarrow  H^{\ast}(A)$ and obtain a morphism $\phi^0:
C^{0,\ast}\overset{\phi'}\rightarrow ZA^{\ast}\hookrightarrow A^{\ast}.$ Define also a
morphism $$ \phi^1: C^{-1,\ast}\rightarrow A^{\ast-1}
$$ by $d\phi^1=\phi^0d^1.$

Assume by induction that we have constructed morphisms
$$
d^r:C^{\ast,\ast}\rightarrow C^{\ast+r,\ast-r+1} \ \  \text{and} \ \
\phi^{r}:C^{-r,\ast}\rightarrow A^{\ast-r}
$$
for $0\leq r\leq n$ (with $d^0=0$) such that
$$
\sum_{k+\ell\leq n+1}d^kd^{\ell}=0 \ \ \text{and} \ \ d\phi^{(n)}=\phi^{(n-1)}d^{(n)}\
\  \text{on} \  \  C_{(-n)}
$$
where $d^{(n)}=\sum_{1\leq r\leq n} d^r$ and $\phi^{(n)}=\sum_{0\leq r\leq n} \phi^r:
C_{(-n)}\rightarrow A .$

Consider the composition $\phi^{(n)}d^{(n)}:C^{-n-1,\ast}\rightarrow A^{\ast-n}.$
Clearly,
$$\phi^{(n)}d^{(n)}: C^{-n-1,\ast}\rightarrow
ZA^{\ast-n} \, (\hookrightarrow A^{\ast-n}).$$ Form the composition $\nu
\phi^{(n)}d^{(n)}: C^{-n-1,\ast}\rightarrow H^{\ast-n}(A)$ to obtain a morphism
$d^{n+1}:C^{-n-1,\ast}\rightarrow C^{0,\ast-n}$ such that $\rho^0d^{n+1}=-\nu
\phi^{(n)}d^{(n)}.$ Since $H^i(C^{i,\ast},d^1)=0$ for $i<0,$ we can extend $d^{n+1}$ on
$C^{\ast,\ast}$ with $\sum_{k+\ell\leq n+2}d^kd^{\ell}=0. $ Then $\nu
\phi^{(n)}d^{(n+1)}=0,$ and there is a morphism $\phi^{n+1}:C^{-n-1,\ast}\rightarrow
A^{\ast-n-1}$ with $d\phi^{(n+1)}=\phi^{(n)}d^{(n+1)}.$

Define $$  d^{\ast}=\sum_{r\geq 1}d^r   \ \ \text{and} \ \ \phi=\sum_{r\geq 0}\phi^r
$$ to obtain the homological resolution $\phi:(C,d^{\ast})\rightarrow
(A,d).$

\end{proof}

In particular, one can take $d^r=0,\,r\geq 2,$ on  $C^{\ast,\ast}$ and $\phi=\rho$ when
$d_{A}=0.$

Note that in the abelian category of modules homological multicomplex resolutions were
in fact constructed in \cite{berika1}, \cite{berika2}(compare \cite{Heller}).

\subsection{Proof of Proposition \ref{Whitehead}} As above the proof uses
the induction on the resolution degree of the homological projective multicomplex $C.$
We assume that $A$ and $B$ are chain complexes; the case of homological multicomplexes
is similar. First show that $f_{_{\#}}$ is an epimorphism. Let $\bar{g}:C\rightarrow
B.$ Consider the restriction $\bar{g}^0=\bar{g}|_{C^{0,\ast}}: C^{0,\ast}\rightarrow
B^{\ast}.$ Since $\bar{g}$ is chain, $\bar{g}^0$ factors through cocycles
$\bar{g}^0:C^{0,\ast}\rightarrow ZB^{\ast}(\hookrightarrow B^{\ast}).$ Since $H(f)$ is
an isomorphism, we can define ${g}^0:C^{0,\ast}\rightarrow ZA^{\ast}\hookrightarrow
A^{\ast}$ such that $\nu fg^0=\nu\bar{g}^0:C^{0,\ast}\rightarrow
ZB^{\ast}\overset{\nu}\rightarrow  H^{\ast}(B).$ Obviously, there is
$s^0:C^{0,\ast}\rightarrow B^{\ast-1}$ with $fg^0-\bar{g}^0=ds^0$ and then put
$\bar{\bar{g}}=\bar{g}+ds^0+s^0d$ to obtain the commutative diagram
$$
\begin{array}{ccccc}
  C_{(0)}(=C^{0,\ast}) & \overset{g^{0}}\longrightarrow & A \\
\downarrow  & &    \ \ \downarrow f  \\
C & \overset{\bar{\bar{g}}}\longrightarrow & B.
\end{array}
$$
Assume by induction that we have constructed morphisms
$$ g^i:C^{-i,\ast}\rightarrow A^{\ast-i},0\leq i\leq n,\ \  \text{and} \ \
\tilde{g}:C\rightarrow B  $$ such that   $\tilde{g}\simeq \bar{g}$ and the following
diagram
$$
\begin{array}{ccccc}
  C_{(n)} & \overset{g^{(n)}}\longrightarrow & A \\
\downarrow  & &    \ \ \downarrow f  \\
C & \overset{{\tilde{g}}}\longrightarrow & B
\end{array}
$$
commutes. Since the above diagram is commutative and  $H(f)$ is an isomorphism we can
choose $g^{n+1}:C^{-n-1,\ast}\rightarrow A^{\ast-n-1}$ together with
$s^{n+1}:C^{-n-1,\ast}\rightarrow B^{\ast-n-2}$ such that  $dg^{n+1}=g^{(n)}d^{(n+1)}$
and $fg^{n+1}-\tilde{g}^{n+1}=ds^{n+1}.$ Put
$\tilde{\tilde{g}}=\tilde{g}+ds^{n+1}+s^{n+1}d$ to obtain the commutative diagram
$$
\begin{array}{ccccc}
  C_{(n+1)} & \overset{g^{(n+1)}}\longrightarrow & A \\
\downarrow  & &    \ \ \downarrow f  \\
C & \overset{\tilde{\tilde{g}}}\longrightarrow & B.
\end{array}
$$
Thus,  $g=\sum_{i\geq 0}g^i:C\rightarrow A$ is a chain map with $fg\simeq \bar{g},$
i.e. $f_{_{\#}}[g]=[\bar{g}].$

Now let $g,h:C\rightarrow A$ be two morphisms such that $fg$ and $fh$  are connected by
a chain homotopy $s:C\rightarrow B,$ i.e. $fg\underset{s}\simeq fh.$ Clearly,
$fg^0-fh^0=ds^0$ for $s^0=s|_{C^{0,\ast}},$ and, since  $H(f)$ is an isomorphism there
is $t^0:C^{0,\ast}\rightarrow A^{\ast-1}$ with $g^0-h^0=dt^0.$ Choose $t^0$ with
$ft^{0}-{s}^{0}=d\beta^0$ for some $\beta^0:C^{0,\ast}\rightarrow B^{\ast-2}.$ Put
${h}^{\prime}=h+dt^0+t^0d$  and $s^{\prime}=\{{s^{\prime}}^k\}_{k\geq 0},$
$${s^{\prime}}^k= \left\{
\begin{array}{lll}

 0, & k=0\\

s^1+\beta^0d, & k=1\\

s^k, &  k>1.
\end{array}
\right.
$$
Then ${h^{\prime}}^{0}=g^{0}$ and $fg\underset{{s}^{\prime}}\simeq fh^{\prime}.$

Assume by induction that we have constructed a morphism $\bar{h}:C\rightarrow A$
together with chain homotopy $ \bar{s} :C\rightarrow B$ such that
$\bar{h}^{(n-1)}=g^{(n-1)},\, \bar{h}\simeq h$ and $fg\underset{\bar{s}}\simeq
f\bar{h}$
 with $\bar{s}^{(n-1)}=0.$
Since $H(f)$ is an isomorphism there is $t^{n}:C^{-n,\ast}\rightarrow A^{\ast-n-1}$
such that $g^{n}-\bar{h}^n=dt^{n}.$ We can choose $t^{n}$  with
$ft^{n}-\bar{s}^{n}=d\beta^{n}$ for some $\beta^{n}:C^{-n,\ast}\rightarrow
B^{\ast-n-2}.$ Put $\bar{\bar{h}}=\bar{h}+dt^{n}+t^{n}d$   and
$\bar{\bar{s}}=\{\bar{\bar{s}}^k\}_{k\geq 0},$
$$\bar{\bar{s}}^k= \left\{
\begin{array}{lll}

 0, & 0\leq k\leq n\\

\bar{s}^{n+1}+\beta^nd, & k=n+1\\

\bar{s}^k, &  k>n+1.
\end{array}
\right.
$$
Then $\bar{\bar{h}}^{(n+1)}=g^{(n+1)}$ and $fg\underset{\bar{\bar{s}}}\simeq
f\bar{\bar{h}}.$ The induction step is completed.

Finally, we get that $g\simeq h$ as required.

\subsection{Proof of Theorem 1} Given a chain complex $A,$ apply Proposition
\ref{resolution} to obtain a resolution $\phi: C\rightarrow A.$ Given a chain map
$f:A\rightarrow A',$ consider a diagram
$$
\begin{array}{ccccc}
                                           \hspace{0.65in}         C' \\
                \hspace{0.8in}  \downarrow  {\phi'}  \\
          C\overset{\phi}\rightarrow A \overset{f} \rightarrow  A'
\end{array}
$$
 and apply Proposition
\ref{Whitehead}  for the quasi-isomorphism $\phi'$  to obtain a multicomplex map
$g:C\rightarrow C'$ such that $[\phi'g]=[f\phi]$ in $K_{\mathcal{M}}(\mathcal{A}).$
Thus, we get
 the functor
$$\varrho: K(\mathcal{A})\rightarrow
K_{\mathcal{P}}(\mathcal{A})$$ which to each chain complex assigns its homological
resolution. Again by the above propositions we deduce that $\varrho$ transforms
quasi-isomorphisms into isomorphisms, so that using the universal property of the
localization functor $Q$  we get the functor
$$\bar{\varrho}: D(\mathcal{A})\rightarrow
K_{\mathcal{P}}(\mathcal{A})$$ such that the diagram
$$
\begin{array}{ccccc}
  K(\mathcal{A})  \overset{Q}\longrightarrow  D( \mathcal{A}) \\
      \varrho\searrow  \ \ \       \ \ \downarrow \bar{\varrho}  \\
          \ \ \ \  K_{\mathcal{P}}(\mathcal{A})
\end{array}
$$
commutes.

 Now it is straightforward to check
that $\bar{\varrho}$ is an inverse for $\iota.$
\subsection{Derived functors} Let $F:\mathcal{A}\rightarrow
\mathcal{B}$ be an additive covariant functor   between abelian categories with enough
projectives and countable coproducts. Obviously, we have the induced functor
$\mathcal{F}:K_{\mathcal{P}}(\mathcal{A})\rightarrow K_{\mathcal{M}}(\mathcal{B}).$ It
is easy to verify that the composition
\[ D(\mathcal{A})\overset{\bar{\varrho}}\longrightarrow
K_{\mathcal{P}}(\mathcal{A}) \overset{\mathcal{F}}\longrightarrow
K_{\mathcal{M}}(\mathcal{B})\overset{\kappa}\longrightarrow
K(\mathcal{B})\overset{Q}\longrightarrow D(\mathcal{B})
\]
is the \emph{left derived functor} in the sense of Verdier
\[LF: D(\mathcal{A})\rightarrow D(\mathcal{B}).\]

\subsection{The minimality of homological resolutions}
Finally, some remarks about the minimality of homological resolutions. For example,
given a chain  complex $(H^{\ast},d)$ on the category of modules over  a principal
ideal domain a homological resolution $(C^{\ast,\ast},d^{\ast})\rightarrow
(H^{\ast},d)$ of $(H^{\ast},d)$ can be chosen to be concentrated in the resolution
degrees $0$ and $-1$ with $d^r=0$ unless $r=1.$ On the other hand, additional
structures on $C^{i,j}$ mentioned in the introduction may impose $i<-1$ (cf.
\cite{sane2}): Namely, if $H^{\ast}=\mathbb{Z}[x_1,...,x_n],n>1,$ is a polynomial
algebra with $d=0$, then the requirement that $C^{\ast,\ast}$ is endowed with a
non-commutative multiplication \emph{compatible} with the bigrading imposes the
multiplicative generators of the minimal resolution $C^{i,j}$ to be concentrated in
resolution degrees $i\geq -n+1,$ while the resolution lengthes    of groups
$C^{\ast,j}$ for $j\geq 0$ are unbounded. If one introduces a non-commutative operation
$\smile_1$ on $C^{\ast,\ast}$ that measures the non-commutativity of the above
multiplication and is compatible with the bigrading, then even the multiplicative
generators can not be no longer chosen to be bounded by the resolution degree and so
on.

\end{document}